\documentclass[a4paper]{article}

\pdfoutput=1				

\usepackage[T1]{fontenc}
\usepackage[utf8]{inputenc}
\usepackage[english]{babel}
\usepackage{amsmath}			
\usepackage{amsfonts}			

\usepackage{amsthm}			
\usepackage{amssymb}			
\usepackage{latexsym}			
\usepackage{bbold}			
\usepackage[stable]{footmisc}
\usepackage{units}
\usepackage{mathtools}			
\usepackage[numbers]{natbib}
\usepackage[table]{xcolor}		
\usepackage{comment}			
\usepackage[textfont=it]{caption}	
\usepackage[breaklinks]{hyperref}			
\usepackage{draftwatermark}		

\SetWatermarkText{Preprint}					
\SetWatermarkScale{5}						
\newenvironment{keywords}{\begin{trivlist}\item[]{\bfseries\sffamily Keywords:}\ } {\end{trivlist}}	
\newenvironment{subject}{\begin{trivlist}\item[]{\bfseries\sffamily MSC:}\ } {\end{trivlist}}
\newenvironment{class}{\begin{trivlist}\item[]{\bfseries\sffamily ACM:}\ } {\end{trivlist}}

\definecolor{grey}{gray}{0.9}

\DeclareMathOperator{\R}{\mathbb{R}}					
\DeclareMathOperator{\1}{\mathbb{1}}					
\DeclareMathOperator{\tr}{tr} 						
\DeclareMathOperator{\diag}{diag} 					

\newtheoremstyle{thm}{15pt}{5pt}{}{}{\bf}{}{0.5em}{}
\theoremstyle{thm}

\newtheorem*{mytheorem*}{Theorem}
\newtheorem*{mydefine*}{Definition}
\newtheorem*{mylemma*}{Lemma}
\newtheorem*{mycorollary*}{Corollary}
\newtheorem*{mynote*}{Note}

\usepackage{etoolbox}
\patchcmd{\thebibliography}{\clubpenalty4000}{\clubpenalty10000}{}{}
\patchcmd{\thebibliography}{\widowpenalty4000}{\clubpenalty10000}{}{}

\newcommand{\Ref}[2]{\mbox{\hyperref[#2]{#1~\ref*{#2}}}}

\title{Cross-Gramian-Based Combined State and Parameter Reduction for Large-Scale Control Systems}
\author{Christian Himpe\thanks{Contact: \href{mailto:christian.himpe@uni-muenster.de}{\nolinkurl{christian.himpe@uni-muenster.de}}, \href{mailto:mario.ohlberger@uni-muenster.de}{\nolinkurl{mario.ohlberger@uni-muenster.de}}, Institute for Computational and Applied Mathematics at the University of M\"unster, Einsteinstrasse 62, D-48149 M\"unster, Germany} \and Mario Ohlberger\footnotemark[1]}
\date{} 

\begin{document}

\setlength{\parindent}{0pt}


\maketitle

\begin{abstract}
\textbf{
This work introduces the empirical cross gramian for multiple-input-multiple-output systems.
The cross gramian is a tool for reducing the state space of control systems, which conjoins controllability and observability information into a single matrix and does not require balancing.
Its empirical gramian variant extends the application of the cross gramian to nonlinear systems.
Furthermore, for parametrized systems, the empirical gramians can also be utilized for sensitivity analysis or parameter identification and thus for parameter reduction.
This work also introduces the empirical joint gramian, which is derived from the empirical cross gramian.
The joint gramian not only allows a reduction of the parameter space, but also the combined state and parameter space reduction, which is tested on a linear and a nonlinear control system.
Controllability- and observability-based combined reduction methods are also presented, which are benchmarked against the joint gramian.
}
\end{abstract}

\begin{keywords}
 Combined Reduction, Model Reduction, Empirical Cross Gramian, Joint Gramian, Controllability, Observability
\end{keywords}

\begin{subject}
 93B11, 93B30, 93C10
\end{subject}

\begin{class}
 G.1.3
\end{class}

\section{Introduction}
The evaluation of large-scale dynamical systems, which arise for example from complex networks or discretized partial differential equations, may require model reduction due to limitations in computing power or memory.
A reduction of the state space generates a surrogate model resembling the same dynamics up to a small error.
For parametrized systems, the model order reduction has to take into account the associated parameter space to ensure the validity of the reduced order model.
If the parameter space is of high dimension, a repeated evaluation at various locations of the parameter space, for example during optimization of inverse problems, may also necessitate a model reduction, yet for the parameter space.
This contribution is concerned with combined state and parameter reduction, targeting models with high dimensional state and parameter spaces.

The efficient reduction of large-scale nonlinear control systems is a challenging task.
Even more in the case of parametrized systems, with high-dimensional state and parameter spaces, where a combined reduction of parameters and states may be required to allow repeated evaluation.
For instance, an inverse problem on a neural network with many nodes and unknown connectivity, modeled as a parametrized nonlinear control system, requires long times during parameter estimation due to system size and parameter count.
Large-scale neural networks have widespread use, such as forward control problems on artificial neural networks or as inverse problems on biological neural networks.
A real-life example is the reconstruction of connectivity between brain regions from activity measurements like EEG or fMRI (see for example \cite{moran07}).

To lower computational complexity, the parameter and state spaces are to be confined to low-dimensional subspaces without affecting the systems dynamics significantly.
Projection-based model order reduction techniques are concerned with determining projections to such subspaces, mapping the high-dimensional model to a reduced order low-dimensional surrogate model.

The methods presented in this work are rooted in balanced truncation \cite{moore81} and proper orthogonal decomposition (POD) \cite{kunisch99}.
Alternative to the here presented method using empirical gramians, another class of balancing-related approaches focuses on solving Lyapunov and Sylvester equations (see for example \cite{benner04}).
For parametrized systems, also the reduced-basis method \cite{haasdonk08} should be noted here.

Since the number of a systems inputs and outputs usually remains fixed, the maps to and from the intermediary states characterizes the reducibility of a system \cite{vandooren00}.
The balanced truncation approach, introduced in \cite{moore81}, balances a system in terms of controllability and observability,
where controllability quantifies how well a state is driven by the input and observability quantifies how well changes in a state are reflected in the output.
Excluding the least controllable and observable states by truncating the balanced system, a reduced order mapping from inputs to outputs is approximated.

A linear time-invariant control system is composed of a linear dynamic system and a linear output transformation,
\begin{align*}
  \dot{x}(t) &= Ax(t) + Bu(t), \\
        y(t) &= Cx(t),
\end{align*}
with states $x(t) \in \R^n$, input or control $u(t) \in \R^m$, and outputs $y(t) \in \R^o$.
The system matrix $A \in \R^{n \times n}$ transforms the states, the input matrix $B \in \R^{n \times m}$ introduces external input or control and the output matrix $C \in \R^{o \times n}$ transforms the states to the outputs.

Controllability and observability can be assessed through the associated controllability gramian $W_C := \int_0^\infty e^{At} B B^T e^{A^T t} dt$ and observability gramian \linebreak $W_O := \int_0^\infty e^{A^T t} C^T C e^{At} dt$.
Classically, $W_C$ and $W_O$ are computed as the smallest semi-positive definite solutions of the Lyapunov equations \linebreak $A W_C + W_C A^T = -BB^T$ and $A^T W_O + W_O A = -C^T C$ respectively.
To make a compound statement about controllability and observability, $W_C$ and $W_O$ have to be balanced \cite{laub87}. 
The singular values of the resulting balanced gramian correspond to the Hankel singular values of the system, with their magnitude describing how controllable and observable the associated state is.

This work focuses on cross gramian-based methods for model reduction, which combines controllability and observability information into one gramian and is elaborately described in \cite{antoulas05}.
The cross gramian $W_X := \int_0^\infty e^{At} BC e^{At} dt$ was introduced in \cite{fernando83a} and corresponds to a solution of the Sylvester equation \linebreak $A W_X + W_X A = -BC$. 

An alternative to solving the Lyapunov or Sylvester matrix equations, apart from the analytic approaches for example in \cite{ionescu11}, is the method of empirical gramians, which was introduced in the works \cite{lall99}, \cite{lall02} and enables the computation of gramian matrices also for nonlinear systems by mere basic vector and matrix operations.
This concept was extended among others in \cite{hahn02a} providing more general input signals.
Particularly noted should be \cite{streif06} and \cite{streif09} for developing the empirical cross gramian for single-input-single-output (SISO) systems in the context of sensitivity analysis.

In this article the empirical cross gramian is generalized to be applicable to multiple-input-multiple-output (MIMO) systems.
For the gramian-based parameter reduction, the groundwork has been laid by \cite{geffen08} from the observability and by \cite{sun06a} from the controllability point of view.
From the cross gramian perspective of parameter reduction, a new gramian, namely the joint gramian, is introduced in this work.
Furthermore, the concept of gramian-based combined state and parameter reduction is established.
Using empirical gramians, it is shown, that combined reduction allows efficient model order reduction of linear and nonlinear control systems.

To begin, the cross gramian and its properties are reviewed in \Ref{section}{crossgramian}.
Next, the empirical cross gramian for MIMO systems is developed in \Ref{section}{empirical}. \linebreak
\Ref{Section}{combined} introduces combined state and parameter reduction in two variants.
First, an observability- and second, a controllability-based approach; the former is enhanced to a cross gramian-based combined reduction, which is presented in \Ref{section}{joint}. 
Finally, numerical experiments are conducted in \Ref{section}{numerics} comparing the newly presented methods for a linear and nonlinear neural network as well as a nonlinear benchmark problem.

\section{Review of the Cross Gramian}\label{crossgramian}
A brief review of the cross gramian along with its application to model reduction of linear time-invariant control systems is given next.
The cross gramian\footnote{also known by the symbol: $W_{CO}$} $W_X$ was introduced in a sequence of works (\cite{fernando83a,fernando82,fernando83b,fernando84a,fernando84b,fernando85,laub83}) and encodes controllability and observability into a single gramian matrix, 
Defined as the product of controllability and observability operator, and it can only be computed for square\footnote{A system with the same number of inputs and outputs} and asymptotically stable systems:
\begin{align}\label{eq:wx}
 W_X := \int_0^\infty e^{A t} B C e^{A t} dt.
\end{align}
Equivalently, the cross gramian is given as a solution to the Sylvester equation $AW_X + W_XA = -BC$.
Approximate solutions for the Sylvester equation were discussed in \cite{sorensen01}, \cite{sorensen02} and \cite{baur08}.
If the system is also symmetric, the following relation between the cross, the controllability and observability gramian holds \cite{fernando83a}:
\begin{align*}
 W_X^2 = W_C W_O \Rightarrow |\lambda(W_X)| = \sqrt{\lambda(W_C W_O)}.
\end{align*}
While a SISO system is always symmetric \cite{fernando83a}, a linear MIMO system not only requires the same number of inputs and outputs,
but also the system gain \linebreak $G=-CA^{-1}B$ has to be symmetric \cite{sorensen01}; then a symmetric transformation $J$, with $AJ = JA^T$ and $B=C^TJ$ exists.
Trivially, for $J=\1$ the system would be restricted by $A=A^T$ and $B=C^T$; such a system is called state-space symmetric.

As presented in \cite{fernando83a}, the trace of the cross gramian equals half the gain for a SISO system $(A,b,c)$ where now $b \in \R^{n \times 1}$ and $c \in \R^{1 \times n}$:
\begin{align}\label{trace}
 \tr(W_X) = -\frac{1}{2}cA^{-1}b.
\end{align}
Because the trace equals the sum of eigenvalues, the cross gramians eigenvalues are associated with the system gain (\ref{trace}).
This result was used in \cite{streif06} and \cite{streif09} for parameter identification purposes, using the system gain as a sensitivity measure.
An extension of (\ref{trace}) from \cite[Theorem 3]{fernando83a} for MIMO systems is developed\footnote{see also \cite{gheondea99}} next:
\begin{mycorollary*}{} ~\\
Given a linear, square, asymptotically stable MIMO system, then the trace of the cross gramian relates to the system gain as follows:
\begin{align*}
 \tr(W_X) = -\frac{1}{2}\tr(CA^{-1}B).
\end{align*}
\end{mycorollary*}
\begin{proof} ~\\
For an asymptotically stable system, the trace of the cross gramian, in the form of (\ref{eq:wx}), is given by:
{\allowdisplaybreaks
\begin{align*}
 \tr(W_X) &= \tr(\int_0^\infty e^{A t} B C e^{A t} dt) \\
          &= \int_0^\infty \tr(e^{A t} B C e^{A t}) dt \\
          &= \int_0^\infty \tr(C e^{A t} e^{A t} B) dt \\
          &= \tr (\int_0^\infty C e^{2A t} B dt) \\
          &= \tr(C \int_0^\infty e^{2 A t} dt B) \\
          &= \tr(C (-\frac{1}{2} A^{-1}) B) \\
          &= -\frac{1}{2}\tr(C A^{-1} B).
\end{align*}}
\end{proof}

Employing the cross gramian instead of controllability and observability gramian, means only a single gramian has to be computed.
And since no balancing is required, the truncation procedure can be simplified to a direct truncation \linebreak (\cite{aldhaheri91}, \cite[Ch. 12.3]{antoulas05}).
A balancing transformation can be approximated by the singular value decomposition (SVD) of the cross gramian.
The approximated Hankel singular values of the diagonal matrix $D$ are sorted by the controllability and observability of the states.
A projection to a subspace of the state space is then given by truncation of $U$ and $V$:
\begin{align}\label{dt}
 &W_X \stackrel{SVD}{=} UDV = \begin{pmatrix}U_1&U_2\end{pmatrix} \begin{pmatrix}D_1&0\\0&D_2\end{pmatrix} \begin{pmatrix}V_1\\V_2\end{pmatrix}.
\end{align}
The matrices $V, U \in \R^{n \times n}$ are partitioned based on a threshold \linebreak $\epsilon \leq 2 \sum_{k=r+1}^n D_{2,kk}$ into $U_1 \in \R^{n \times r}$, $U_2 \in \R^{n \times (n-r)}$ and $V_1 \in \R^{r \times n}$, \linebreak $V_2 \in \R^{(n-r) \times n}$.
This leads to the following reduced order model\footnote{Here the (one-sided) Galerkin projection is used, since the (two-sided) Petrov-Galerkin projection may produce unstable reduced order models.}:
\begin{align*}
 &\tilde{A} =  U_1^T A U_1, \quad \tilde{B} = U_1^T B, \quad \tilde{C} = C U_1, \quad \tilde{x}_0 = U_1^T x_0, \\
 \Rightarrow &\begin{cases} \dot{\tilde{x}}(t) = \tilde{A} \tilde{x}(t) + \tilde{B} u(t), \\ \tilde{y}(t)  = \tilde{C} \tilde{x}(t). \end{cases}
\end{align*}
Apart from truncation-based model reduction, the cross gramian has applications, for example, in system identification \cite{ionescu11} and decentralized control (\cite{moaveni06,moaveni08,moaveni08b}) by computing a participation matrix (see \cite{shaker12}) based on the cross gramian.
Lastly, the cross gramian also has the benefit of conveying more information than controllability and observability gramian, since the system's Cauchy index is given by the cross gramian's signature \cite{fernando83b}.

\section{Empirical Cross Gramian}\label{empirical}
In this section the empirical cross gramian for MIMO systems is introduced.
For general, possibly nonlinear, control systems of the form:
\begin{align}\label{nsys}
 \dot{x}(t) &= f(x(t),u(t)), \\
 y(t) &= g(x(t),u(t)), \notag
\end{align}
with states $x(t) \in \R^n$, input or control $u(t) \in \R^m$, outputs $y(t) \in \R^o$, a vector field $f:\R^n \times \R^m \to \R^n$ and an output function $g:\R^n \times \R^m \to \R^o$, the procedure from Section \ref{crossgramian} is not viable.
In \cite{lall99}, \cite{lall02} and \cite{hahn00} the concept of empirical (controllability and observability) gramians was introduced.
This is a POD method based solely on state-space simulations of the system \cite{antoulas01}.
These empirical gramians correspond to the classic gramians for linear systems as shown in \cite{lall99}. 
Subsequently this approach and its field of application was advanced by \cite{hahn02b}, \cite{hahn03}, \cite{condon04} and \cite{shaker12}.
Because the empirical gramians can be aligned to the operating region of the underlying system in terms of initial states and input or control, the empirical gramians carry more detailed information on the system \cite{singh05} than the gramians computed as solutions of matrix equations. 

Empirical gramians are based on averaging the response of a system that is perturbed in inputs and initial states. 
Initially, the perturbed input was restricted to a delta impulse $u(t) = \delta(t)$, which was broadened to more general input configurations in \cite{hahn02a} under the name of empirical controllability covariance matrix and empirical observability covariance matrix.

The necessary perturbation sets are systematically defined next; these should reflect the operating range of the underlying system.
$E_u$ and $E_x$ are sets of standard directions for the inputs and initial states. 
Sets $R_u$ and $R_x$ are orthogonal transformations (rotations) to these standard directions of inputs and initial states respectively, while $Q_u$ and $Q_x$ hold scales to these directions:
\begin{align*}
 E_u &= \{ e_i \in \R^j ; \|e_i\| = 1 ; e_i e_{j \neq i} = 0; i=1,\ldots,m \}, \\
 E_x &= \{ f_i \in \R^n ; \|f_i\| = 1 ; f_i f_{j \neq i} = 0; i=1,\ldots,n \}, \\
 R_u &= \{ S_i \in \R^{j \times j} ; S_i^T S_i = \1 ; i = 1,\ldots,s \}, \\
 R_x &= \{ T_i \in \R^{n \times n} ; T_i^T T_i = \1 ; i = 1,\ldots,t \}, \\ 
 Q_u &= \{ c_i \in \R ; c_i > 0 ; i = 1,\ldots,q \}, \\ 
 Q_x &= \{ d_i \in \R ; d_i > 0 ; i = 1,\ldots,r \}.
\end{align*}
Along the lines of the empirical controllability gramian and empirical observability gramian \cite{lall99}, the empirical cross gramian for SISO systems was introduced in \cite{streif06}. 
In this work, as a new contribution, the empirical cross gramian is generalized to square MIMO systems.
Hence the scope of the cross gramian is extended to nonlinear control systems and provides an alternative nonlinear cross gramian to \cite{ionescu11}.
For a general (possibly nonlinear) MIMO system with $\dim(u) = \dim(x)$ the empirical cross gramian is defined by:
\begin{mydefine*}{(Empirical Cross Gramian)}\label{cross} \\
For sets $E_u$, $E_x$, $R_u$, $R_x$, $Q_u$, $Q_x$, input $\bar{u}$ during steady state $\bar{x}$ with output $\bar{y}$, 
the \textbf{empirical cross gramian} $\hat{W}_X$ relating the states $x^{hij}$ of input \linebreak $u^{hij}(t) = c_h S_i e_j \delta(t) + \bar{u}$ to output $y^{klb}$ of $x_0^{klb} = d_k T_l f_b + \bar{x}$, is given by:
 \begin{align*}
  \hat{W}_X &= \frac{1}{|Q_u| |R_u| m |Q_x| |R_x|} \sum_{h=1}^{|Q_u|} \sum_{i=1}^{|R_u|} \sum_{j=1}^m \sum_{k=1}^{|Q_x|} \sum_{l=1}^{|R_x|} \frac{1}{c_h d_k} T_l \int_0^\infty \Psi^{hijkl}(t) d t \; T_l^T, \\
  &\Psi_{ab}^{hijkl}(t)  = f_a^T T_l^T \Delta x^{hij}(t) e_j^T S_i^T \Delta y^{klb}(t) \in \R, \\
  &\Delta x^{hij}(t) = (x^{hij}(t)-\bar{x} ), \\
  &\Delta y^{klb}(t) = (y^{klb}(t)-\bar{y} ).
 \end{align*}
 
\end{mydefine*}
Essentially, the empirical cross gramian is an averaged cross gramian over snapshots with the specified perturbations in input and initial states around steady state input $x(\bar{u})$ and steady state $y(\bar{x})$.

Next, similar to \cite{lall99} and \cite{streif06}, the equality of the cross gramian and the empirical cross gramian for linear MIMO control systems is shown next: \pagebreak

\begin{mylemma*}{(Empirical Cross Gramian)} \\
For any nonempty sets $R_u$, $R_x$, $Q_u$, $Q_x$ the empirical cross gramian $\hat{W}_X$ of an asymptotically stable linear control system is equal to the cross gramian.
\end{mylemma*}

\begin{proof} ~\\
For an asymptotically stable linear control system, the input-to-state and state-to-output maps are given by: 
\begin{align*}
 \Delta x(t) &= x(t) = e^{At} Bu(t), \\
 \Delta y(t) &= y(t) = Ce^{At}x_0,
\end{align*}
thus:
\begin{align*}
  \Psi^{hijkl}_{ab} &= f_a^T T_l^T (e^{At} B c_h S_i e_j) e_j^T S_i^T (C e^{At} d_k T_l f_b) \\
                    &= c_h d_k f_a^T T_l^T e^{At} B C e^{At} T_l f_b, \\
 \Rightarrow \Psi^{hijkl} &= c_h d_k T_l^T e^{At} B C e^{At} T_l, \\
 \Rightarrow \hat{W}_X    &= \frac{1}{|Q_u| |R_u| m |Q_x| |R_x|} \sum_{h=1}^{|Q_u|} \sum_{i=1}^{|R_u|} \sum_{j=1}^m \sum_{k=1}^{|Q_x|} \sum_{l=1}^{|R_x|}  \int_0^\infty e^{At} B C e^{At} d t \\
            &= \int_0^\infty e^{At} B C e^{At} d t \\ 
            &= W_X.
\end{align*}
\end{proof}

As for the other empirical gramians, this proof is only valid for impulse input, yet a similar approach to \cite{hahn02a} can be used to extend the empirical cross gramian, yielding an empirical cross covariance matrix by allowing general discrete input signals.
The snapshots $x^{hij}$ and $y^{klb}$ can be computed as simulations on demand, but also be included from observed experimental data.
In case the data is collected at discrete times $t$ in regular intervals $\Delta t$, following \cite{hahn01}, a discrete representation of the empirical cross gramian is given here, too:

\begin{mydefine*}{(Discrete Empirical Cross Gramian)} \\
For sets $E_u$, $E_x$, $R_u$, $R_x$, $Q_u$, $Q_x$, input $\bar{u}$ during steady state $\bar{x}$ with output $\bar{y}$, 
the \textbf{discrete empirical cross gramian} $\mathcal{W}_X$ relating the states $x^{hij}$ of input $u^{hij}(t) = c_h S_i e_j \delta(t) + \bar{u}$ to output $y^{klb}$ of $x_0^{klb} = d_k T_l f_b + \bar{x}$, is given by:
 \begin{align*}
  \mathcal{W}_X &= \frac{1}{|Q_u| |R_u| m |Q_x| |R_x|} \sum_{h=1}^{|Q_u|} \sum_{i=1}^{|R_u|} \sum_{j=1}^m \sum_{k=1}^{|Q_x|} \sum_{l=1}^{|R_x|} \frac{\Delta t}{c_h d_k} T_l \sum_{t=0}^{\mathfrak{T}} \Psi^{hijkl}_t T_l^T, \\
  &\Psi_{ab,t}^{hijkl}  = f_a^T T_l^T \Delta x^{hij}_t e_j^T S_i^T \Delta y^{klb}_t \in \R, \\
  &\Delta x^{hij}_t = (x^{hij}_t-\bar{x} ), \\
  &\Delta y^{klb}_t = (y^{klb}_t-\bar{y} ).
 \end{align*}
\end{mydefine*}

Computational complexity depends largely on the number of scales and rotations of perturbations as well as the order of integration used to generate the snapshots $x^{hij}$ and $y^{klb}$.

The empirical cross gramian enables state reduction for square nonlinear control systems without an additional balancing procedure using direct truncation, where the approximately balancing projection $U$ is computed by an SVD and truncated to $U_1$, analogous to (\ref{dt}):
\begin{align*}
 \dot{x}(t) &= U_1^T f(U_1 x(t),u(t)), \\
 y(t) &= g(U_1 x(t),u(t)).
\end{align*}
Like for POD, to quantify how close the subspace obtained by reduction is approximating the state space, a measure of total preserved energy \cite{lall99}, \cite{condon04a} can also be employed here:
\begin{align*}
 \tilde{E} = \frac{\sum_{i=1}^k \sigma_i}{\sum_{i=1}^n \sigma_i},
\end{align*}
for $k$ retained states of a $n$-dimensional model with $n-k$ truncated states related to the $n-k$ lowest singular values of the empirical cross gramian $\sigma_i(W_X)$.

\section{Combined State and Parameter Reduction}\label{combined}
Two methods for combined reduction, allowing simultaneous reduction of state and parameter spaces, are proposed; an observability-based and a controllability-based ansatz.
For parametrized general control systems with parameters $\theta \in \R^p$:
\begin{align*}
 \dot{x}(t) &= f(x(t),u(t),\theta), \\
 y(t) &= g(x(t),u(t),\theta),
\end{align*}
in \cite{geffen08}, the identifiability gramian was introduced\footnote{In \cite{singh05a} a similar concept is presented}, which extends the concept of observability from states to parameters.
Augmenting the states of a given system by its parameters $\theta$ as constant components, the parameters are treated like states:
\begin{align} \label{aug1}
 \dot{\breve{x}}(t) &= \begin{pmatrix} \dot{x}(t) \\ \dot{\theta}(t) \end{pmatrix} = \begin{pmatrix} f(x(t),u(t),\theta) \\ 0 \end{pmatrix}, \notag \\
 y(t) &= g(x(t),u(t),\theta), \\
 \breve{x}_0 &= \begin{pmatrix} x_0 \\ \theta \end{pmatrix}. \notag
\end{align}
The initial states $\breve{x}_0$ are also augmented by the given parameter value $\theta$, yielding: $\breve{x}(t) \in \R^{n+p}$.
The identifiability of the parameters is obtained through the observability of these parameter-states by the augmented observability gramian:
 \begin{align*}
  \breve{W}_O &= \begin{pmatrix} W_O & \vline & W_M \\ \hline W_M^T & \vline & W_P \end{pmatrix} \in \R^{(n+p) \times (n+p)},
 \end{align*}
with the state observability gramian $W_O \in \R^{n \times n}$, the parameter observability gramian $W_P \in \R^{p \times p}$ and a mixture matrix $W_M \in \R^{n \times p}$.
From the observability gramian of this augmented system, the identifiability gramian $W_I \in \R^{p \times p}$ can be extracted via the Schur-complement,
 \begin{align*}
  W_I &= W_P - W_M^T {W_O}^{-1} W_M.
 \end{align*}
For an approximation of the identifiability gramian $W_I$, the parameter observability gramian $W_P \approx W_I$ itself is often sufficient.
The identifiability is then given as the observability of the parameters, through the singular values of $W_I$, or approximately $W_P$ respectively.
Instead of using the identifiability information for parameter identification as in \cite{geffen08} and \cite{eberle12}, a projection to the dominant parameter subspace is computed from $W_I$. 
Similar to the cross gramian approach, a singular value decomposition of the approximate identifiability gramian yields the reduced parameters $\tilde{\theta}$:
\begin{align*}
 &W_I \stackrel{SVD}{=} \Pi\Delta\Lambda = \begin{pmatrix}\Pi_1\\ \Pi_2\end{pmatrix} \begin{pmatrix}\Delta_1&0\\0&\Delta_2\end{pmatrix} \begin{pmatrix}\Lambda_1&\Lambda_2\end{pmatrix}, \\
 \Rightarrow & \; \tilde{\theta} = \Pi_1^T \theta, \quad \Pi_1 \tilde{\theta} \approx \theta.
\end{align*}
The partitioning depends on the singular values in $\Delta$.
A truncation of the projection $\Pi$ results in the reduced parameters $\hat{\theta}$ and the associated parameter reduced order model:
\begin{align*}
 \dot{\tilde{x}}(t) &= f(x(t),u(t),\Pi_1 \tilde{\theta}), \\
 \tilde{y}(t) &= g(x(t),u(t),\Pi_1 \tilde{\theta}).
\end{align*}
Next, a combined reduction of state and parameter space is introduced.
With the identifiability gramian-based parameter reduction, first, the parameter space of the system is reduced.
The observability of the states is encoded in the augmented observability gramian, too; it can be extracted as the upper-left $n \times n$ matrix from $\breve{W}_O$.
Then, after computation of a controllability gramian $W_C$, the state space is reduced by balanced truncation\footnote{Balanced truncation provides two-sided truncated projection matrices $U_1$ and $V_1$; see \cite{keil03} }.
This results in an observability-based combined state and parameter reduction:
\begin{align}\label{cr}
 \dot{\tilde{x}}(t) &= V_1 f(U_1 x(t),u(t),\Pi_1 \tilde{\theta}), \\
 \tilde{y}(t) &= g(U_1 x(t),u(t),\Pi_1 \tilde{\theta}). \notag
\end{align}
Similarly, a controllability-based combined reduction can be achieved by a parameter reduction using the sensitivity gramian from \cite{himpe13a} for additive partitionable systems:
\begin{align} \label{ws}
 f(x,u,\theta) &= f(x,u) + \sum_{k=1}^p f(x,\theta_k), \notag \\
 \Rightarrow W_C &= W_{C,0} + \sum_{k=1}^p W_{C,k}, \\
 W_S &= \begin{pmatrix} \tr(W_{C,1}) & & 0 \\ & \ddots & \\ 0 & & \tr(W_{C,p}) \end{pmatrix} \in \R^{p \times p}, \notag
\end{align}
which is based on \cite{sun06a} and \cite{sun06b} and treats the parameters as additional inputs.
By the sensitivity gramian, controllability information on the parameters is provided, that also allows a parameter reduction; again by a singular value decomposition of $W_S$.
Since an approximate controllability gramian $W_C$ is also computed in the process (\ref{ws}), after computation of an observability gramian $W_O$, the parameter reduced system is reduced in states by balanced truncation.
This results in a controllability-based combined state and parameter reduction.
The controllability-based combined reduced order model has the same form as the observability-based reduced model (\ref{cr}).

\section{Joint Gramian and Combined Reduction}\label{joint}
In addition to controllability- and observability-based combined reduction, a cross-gramian-based combined reduction is proposed next, which, for symmetric control systems, is enabled by the empirical cross gramian for MIMO systems from \Ref{section}{empirical}.
Aggregating the computation of controllability and observability not only for states like the cross gramian, but also for identifiability of parameters, leads to a reduction and identification method requiring a new single gramian. 
Here, the same augmented system (\ref{aug1}) is used.
The systems symmetry is not affected by the augmentation with the constant (parameter)-states, since in terms of a linear system the system components $\{A,B,C\}$ are expanded with zeros.
This leads to the following new gramian matrix, which utilizes the cross gramian and thus unifies controllability and observability of states and parameters:
\begin{mydefine*}{(Joint Gramian)} \\
The \textbf{joint gramian} $W_J$ is the cross gramian of a square augmented system, see (\ref{aug1}).
\end{mydefine*}
The joint gramian also has a $2 \times 2$ block structure:
\begin{align*}
 W_J = \begin{pmatrix} W_X & \vline & W_M \\ \hline W_m & \vline & W_P \end{pmatrix} \in \R^{(n+p) \times (n+p)},
\end{align*}
with the state cross gramian $W_X \in \R^{n \times n}$, the parameter cross gramian \linebreak[4] $W_P \in \R^{p \times p}$, and the mixture matrices $W_M \in \R^{n \times p}, W_m \in \R^{p \times n}$.
The parameter-states are uncontrollable, yielding $W_P = 0$ and $W_m = 0$, since no inputs affect the augmented states.
Thus a Schur complement of the joint gramian to extract the parameter associated lower right block matrix will always be zero:
\begin{align*}
 W_{\dot{I}} = W_P - W_m W_X^{-1} W_M = 0 - 0\; W_X^{-1} W_M = 0.
\end{align*}
Yet, the identifiability information on the parameters is encoded in the non-zero mixture matrix $W_M$.
By taking the symmetric part of the joint gramian $\overline{W}_J = \frac{1}{2}(W_J+W_J^T)$, one obtains the cross-identifiability gramian $W_{\ddot{I}}$:
\begin{align*}
 W_{\ddot{I}} = 0 - \frac{1}{4} W_M^T \overline{W}_X^{-1} W_M.
\end{align*}
Taking the inverse of the symmetric part of the cross gramian is too costly in a large-scale setting.
But the Schur complement can be approximated by using:
\begin{align*}
 D &:= \diag(\overline{W}_X) \\
 \overline{W}_X^{-1} &\approx \overline{w}_X^{-1} = D^{-1} - D^{-1}(\overline{W}_X-D)D^{-1}
\end{align*}
as a coarse approximation\footnote{This approximation of the inverse is of complexity $n^2$.} to the inverse from \cite{wu13}.
Thus, a more efficient cross-identifiability gramian is given by:
\begin{align*}
 W_{\ddot{I}} = -\frac{1}{4} W_M^T \overline{w}_X^{-1} W_M.
\end{align*}

A reduced set of parameters $\tilde{\theta}$ is again computed by a truncated projection obtained from the singular value decomposition of the cross-identifiability gramian:
\begin{align*}
 &W_{\ddot{I}} \approx W_P \stackrel{SVD}{=} \Pi\Delta\Lambda = \begin{pmatrix}\Pi_1\\ \Pi_2\end{pmatrix} \begin{pmatrix}\Delta_1&0\\0&\Delta_2\end{pmatrix} \begin{pmatrix}\Lambda_1&\Lambda_2\end{pmatrix}, \\
 \Rightarrow & \; \tilde{\theta} = \Pi_1^T \theta, \quad \Pi_1 \tilde{\theta} \approx \theta.
\end{align*}
After a parameter reduction,
\begin{align*}
 \dot{\tilde{x}}(t) &= f(x(t),u(t),\Pi_1 \tilde{\theta}), \\
 \tilde{y}(t) &= g(x(t),u(t),\Pi_1 \tilde{\theta}),
\end{align*}
the states can be reduced with a state reduction by direct truncation, employing the usual cross gramian $W_X$, a byproduct of the joint gramian $W_J$, which is the upper left $n \times n$ block matrix of $W_J$:
\begin{align*}
 \dot{\tilde{x}}(t) &= V_1 f(U_1 x(t),u(t),\Pi_1 \tilde{\theta}), \\
 \tilde{y}(t) &= g(U_1 x(t),u(t),\Pi_1 \tilde{\theta}).
\end{align*}
For this combined reduction of states and parameters no further gramians have to be computed and no balancing transformation is required.

\section{Implementation and Numerical Results}\label{numerics}
For an efficient implementation, the structure of the gramian computation is exploited.
First, the empirical gramians allow extensive parallelization.
Each combination of direction, transformation and scale can be processed separately yielding a sub-gramian.
Second, the assembly of each sub-gramian can be comprehensively vectorized, since it consists of vector additions, inner- and outer-products.
In the special case of the empirical cross gramian, and thus the empirical joint gramian which is an encapsulation of the empirical cross gramian, organizing the observability snapshots into a 3rd-order-tensor and exploiting generalized transpositions results in a very efficient gramian assembly\footnote{see \texttt{emgr.m} }. 
The final resulting gramian is the normalized accumulation over all sub-gramians.
For further details about the implementation see \cite{himpe13a}.

All gramians for the numerical results are computed by the empirical gramian framework introduced in \cite{himpe13a}.
The empirical gramian framework \textbf{emgr} \cite{emgr} can be found at \mbox{\url{http://gramian.de}} and is compatible with \textbf{Octave} \cite{octave} and \textbf{Matlab}\textsuperscript{\textregistered} \cite{matlab}.
The source code, used for the following experiments, can be found at \mbox{\url{http://gramian.de/himpe14a_sourcecode.tgz}}.

The error measure employed in the following experiments is the relative $L_2$-error for a vector valued time series \cite{fortuna12} is defined by:
\begin{mydefine*} ~\\
 The relative $L_2$-error for two vector valued time series $y$ and $\tilde{y}$ is given by:
 \begin{align*}
  \epsilon = \frac{\|y-\tilde{y}\|_2}{\|y\|_2},
 \end{align*}
 with the $L_2$-norm of a (discrete) time series,
 \begin{align*}
  \|y\|_2 = \sqrt{\sum_t \|y(t)\|_2^2}.
 \end{align*}
\end{mydefine*}

Numerical results for three models are presented next.
First, state reduction is applied to a nonlinear benchmark problem to validate the applicability of the empirical cross gramian.
Then, a linear and nonlinear parametrized control system is considered for the state, parameter and combined reduction.

\subsection{Nonlinear Benchmark}
Introduced in \cite{chen99}, this nonlinear benchmark\footnote{This benchmark is also listed in the MORwiki \cite{morwiki}} has been used in \cite{condon04,condon04a,condon07} as a test problem for the assessment of the empirical controllability gramian and empirical observability gramian in a balanced truncation model order reduction setting.
This benchmark system models a circuit consisting of capacitors and nonlinear resistors\footnote{A resistor parallel-connected to a diode.}.
Its mathematical model is given by the following nonlinear SISO control system:
\begin{align*}
\dot{x}(t) &= \begin{pmatrix} -g(x_1(t)) - g(x_1(t) - x_2(t)) \\ g(x_1(t)-x_2(t)) - g(x_2(t)-x_3(t)) \\ \vdots \\ g(x_{k-1}(t) - x_k(t)) - g(x_k(t) - x_{x+1}(t)) \\ \vdots \\ g(x_{N-1}(t) - x_N(t)) \end{pmatrix}+\begin{pmatrix}u(t) \\ 0 \\ \vdots \\ 0 \\ \vdots \\ 0 \end{pmatrix}, \\
      y(t) &= x_1(t),
\end{align*}
with the nonlinear function $g:\R \to \R$,
\begin{align*}
 g(x) = \exp(x)+x-1.
\end{align*}
In this setting with $\dim(x) = 100$, a zero initial state $x_0 = 0$ and a decaying exponential input $u(t) = e^{-t}$ is applied.
Figure~\ref{benchred} shows the relative $L_2$ error in the reduced order models outputs reduced by balanced truncation of empirical controllability and observability gramian and direct truncation of the empirical cross gramian.

\begin{figure}[h!]
\centering
\includegraphics[width=.5\textwidth]{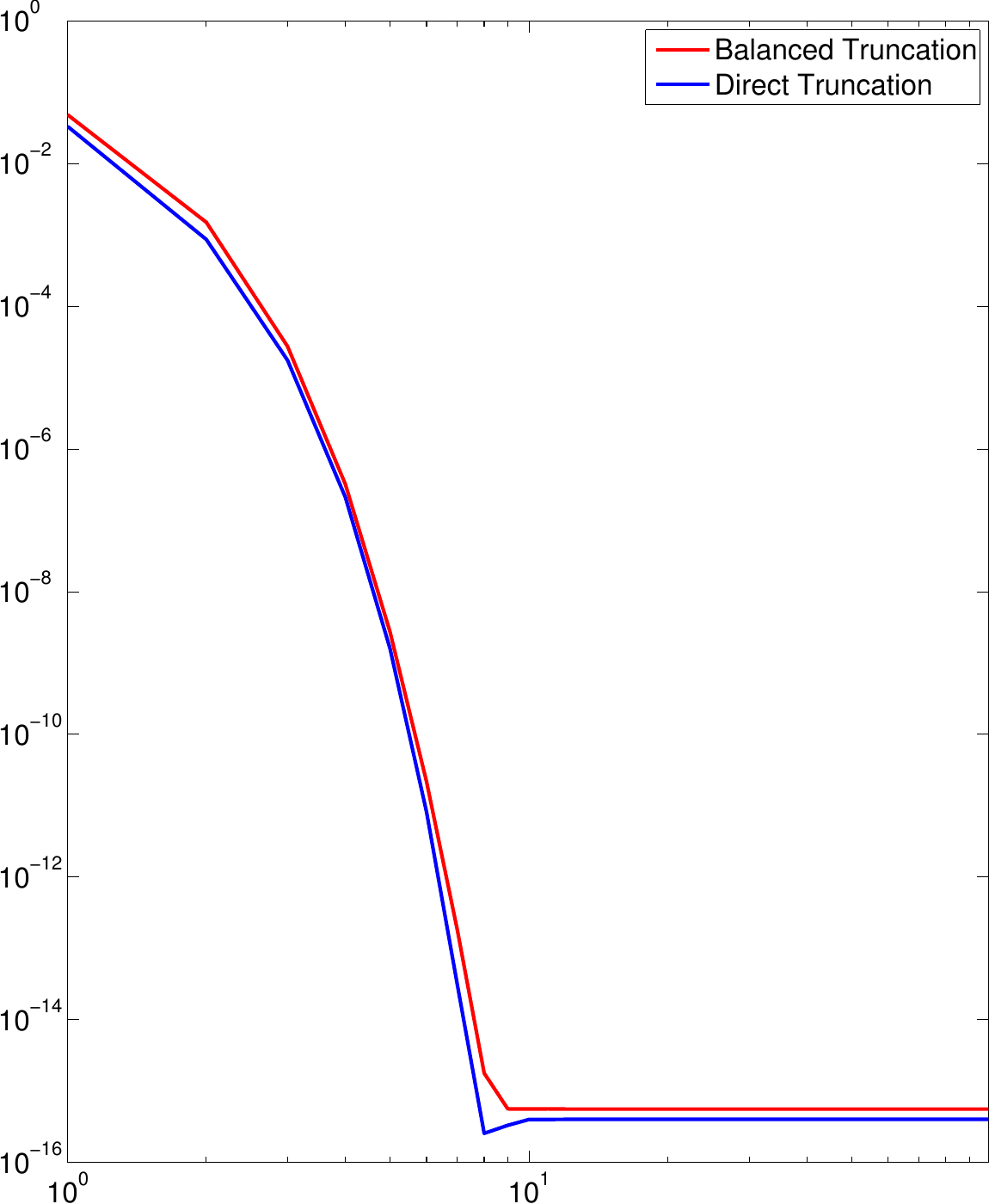}
\caption{Relative $L_2$ output error in the reduced order nonlinear benchmark after the state reduction using balanced truncation and direct truncation of the empirical cross gramian.}
\label{benchred}
\end{figure}

After a steep initial drop in the error, the relative output error remains near the machine precision.
Both methods, balanced truncation and direct truncation, perform very similar and require less then $10\%$ of the full order model's state dimension to reach the error plateau.
The empirical cross-gramian-based direct truncation exhibits a slightly lower output error.

\subsection{State Reduction}
Next, a parametrized linear control system and a parametrized nonlinear control system are reduced in states, parameters and combined in states and parameters.
The parametrized linear control system model \cite{sun06a} is given by:
\begin{align}\label{linmod}
 \dot{x} &= Ax + Bu + \theta, \\
       y &= Cx, \notag
\end{align}
with an additional parametrized source term $\theta \in \R^N$.

The system is ensured to be asymptotically stable, thus $\lambda_{1 \hdots n}(A)<0$, which is a central requirement for all empirical gramians to be computable. 
Furthermore, the system is chosen to be state-space symmetric: $A = A^T$ and $C = B^T$.
A state-space symmetric system provides that all system gramians, the controllability gramian, the observability gramian and the cross gramian are equal \cite{liu98}.
This allows the verification of the empirical cross gramian.

The parametrized nonlinear control system model is given by:
\begin{align}\label{hypmod}
 \dot{x} &= A \tanh(\frac{1}{4}x) + Bu + \theta, \\
       y &= Cx, \notag
\end{align}
with a hyperbolic tangent nonlinearity using the same system matrices $\{A,B,C\}$ as in (\ref{linmod}).
This model is related to the hyperbolic network model from \cite{quan01}.

The linear and nonlinear model are subject to zero initial states $x_0 = 0$ and impulse input $u(t) = \delta(t)$.
For the following experiments a state space dimension $n = \dim(x(t)) = 256$, thus $p = \dim(\theta) = 256$ is assumed as well as an input and output dimension of $m = \dim(u(t)) = o = \dim(y(t)) = 16$.
The parameters are drawn from a uniform distribution $\theta_i = U(0,\frac{1}{10})$; the system matrix $A$ is generated as a sparse uniformly random matrix with ensured stability, the input matrix $B$ is a dense uniformly random matrix, yielding the output matrix \linebreak $C=B^T$. 

The linear model in (\ref{linmod}) and the nonlinear model in (\ref{hypmod}) are first reduced in states using the following methods: balanced truncation utilizing the empirical controllability gramian and empirical observability gramian, and direct truncation of the empirical cross gramian presented in section~\ref{cross}.
Furthermore, a method closely related to balanced POD \cite{willcox02}, \cite{barbagallo08} is tested, too.
For the linear model an approximate cross gramian $W_Y$ is computed using the approach from \cite{fernando85} and \cite{shaker12a}, which obtains the cross gramian by computing the controllability gramian of the system augmented with its adjoint system\footnote{This approach is also implemented in the empirical gramian framework as empirical approximate cross gramian}. 
For the nonlinear model an approximate cross gramian is computed following \cite{willcox02} by the product $W_Y = W_C W_O$.
In figure~\ref{statered} the relative $L_2$-error in the outputs $y$ is plotted for reduced orders $\dim(\tilde{x}) = 1 \hdots n-1$.

\begin{figure}[h!]
\includegraphics[width=\textwidth]{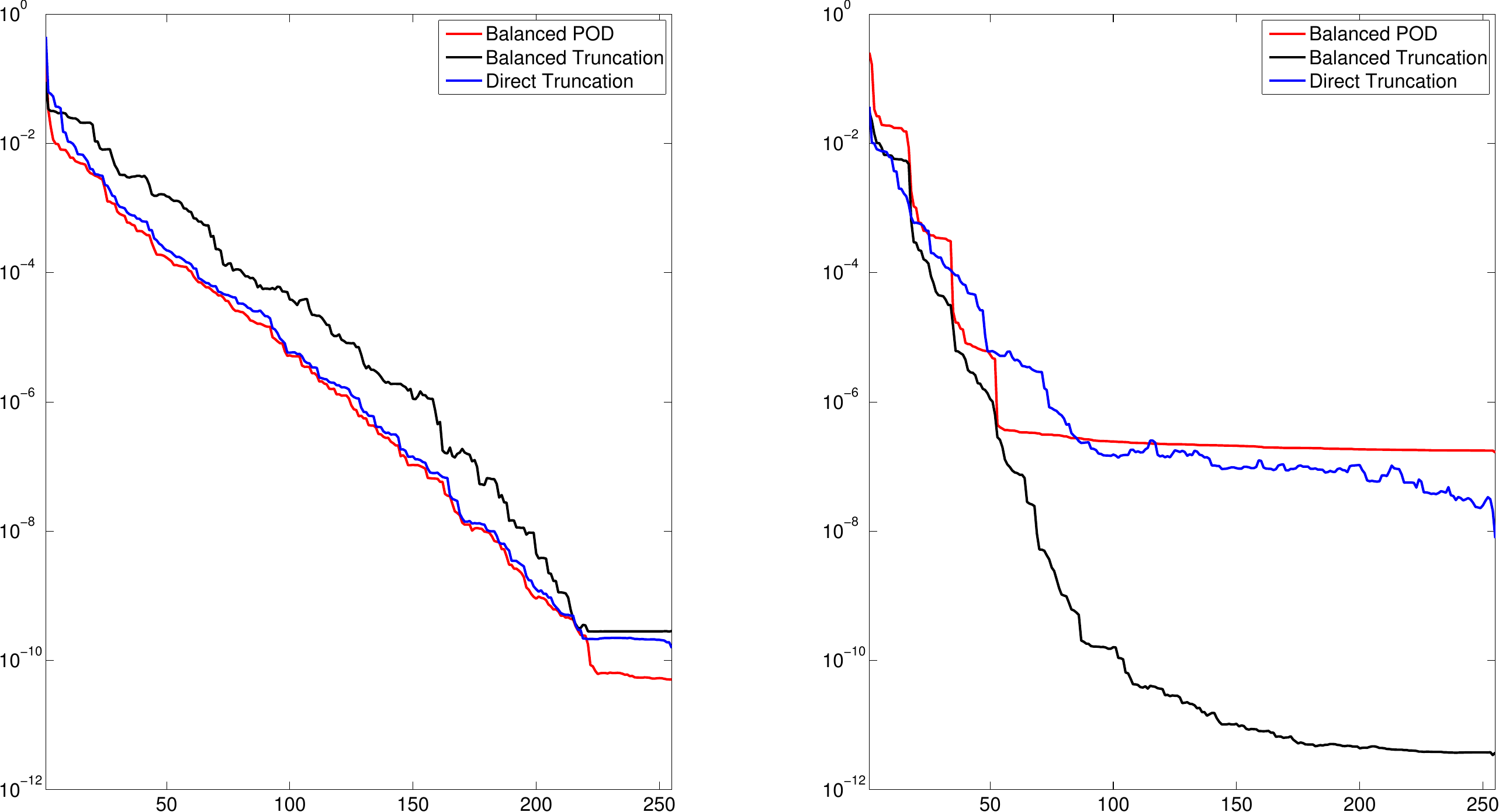}
\caption{Relative $L_2$ output error in the reduced order linear and nonlinear models after the state reduction using balanced POD, balanced truncation and direct truncation of the empirical cross gramian.}
\label{statered}
\end{figure}

The state reduced models in the linear setting, generated by balanced POD, balanced truncation and direct truncation, are of similar quality; 
yet  balanced POD and direct truncation perform slightly better than balanced truncation.
In the nonlinear setting balanced truncation outperforms balanced POD and direct truncation for which the output error flattens above a reduced order of about $40\%$ of the original model order. 

The better performance of balanced truncation for the nonlinear model is due to use of two-sided projections as opposed to the one-sided projections used for balanced POD and direct truncation here.

\subsection{Parameter Reduction}
The linear model in (\ref{linmod}) and the nonlinear model in (\ref{hypmod}) are reduced in parameters using the empirical sensitivity gramian, the empirical identifiability gramian and empirical cross identifiability gramian from the empirical joint gramian.
Figure~\ref{paramred} depicts the relative $L_2$-error in the outputs $y$ for the reduced orders $\dim(\tilde{\theta}) = 1 \hdots n-1$.

\begin{figure}[h!]
\includegraphics[width=\textwidth]{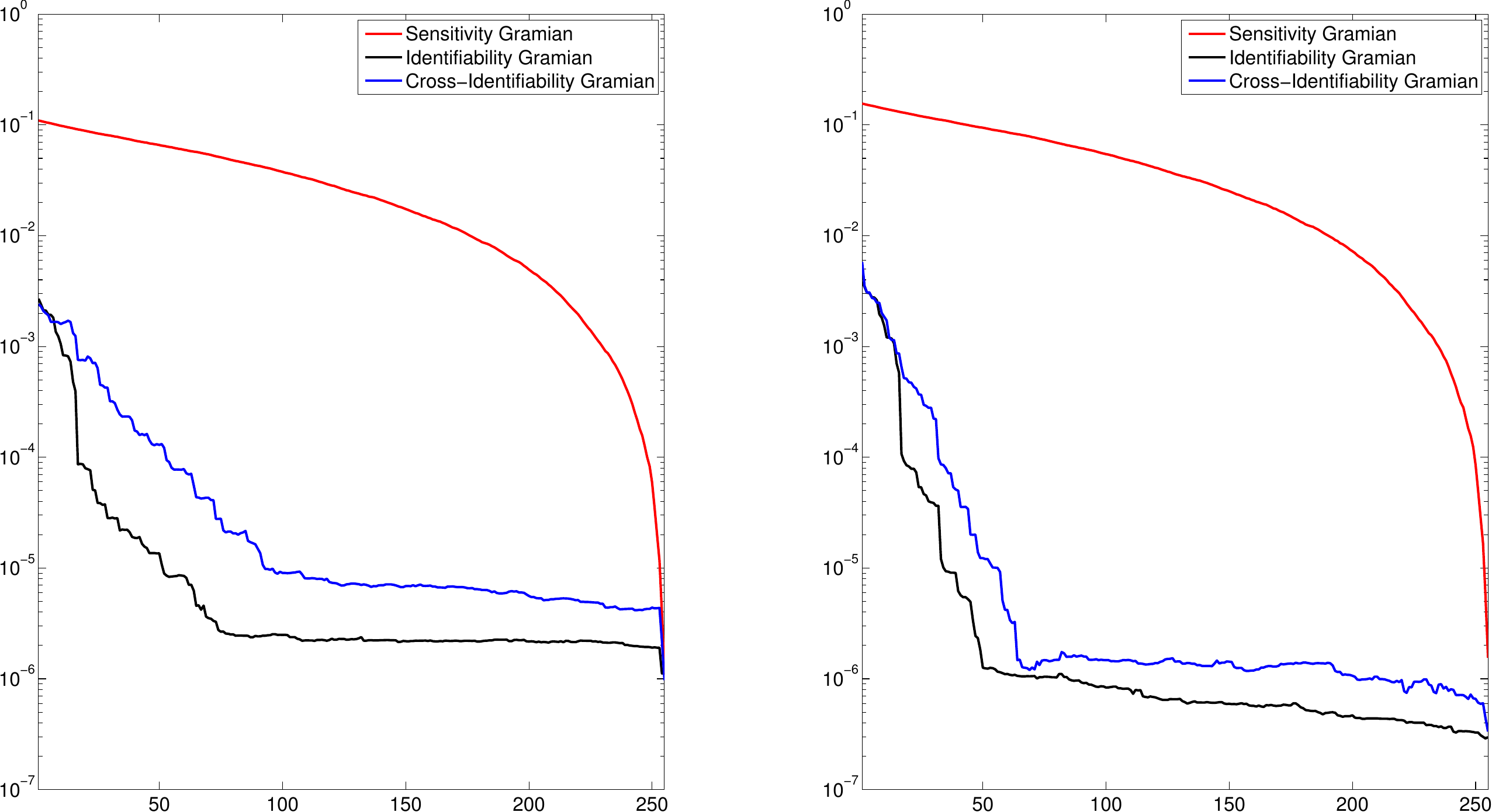}
\caption{Relative $L_2$ output error in the reduced order linear and nonlinear models after the parameter reduction using the sensitivity gramian, identifiability gramian, cross-identifiability gramian.}
\label{paramred}
\end{figure}

For the linear and the nonlinear model, the controllability-based sensitivity gramian performs worst with the slowest decline in output error.
The observa\-bility-based identifiability gramian and the cross-gramian-based cross-identifia\-bility gramian show a sharper descent of the output error of similar quality in the linear and nonlinear setting, 
yet the observability-based parameter reduction exhibits a steeper drop of the output error for reduced orders up to $40\%$ of the original parameter space.

Since the sensitivity gramian is a diagonal matrix, the associated projections reorder the parameters and thus excludes all effects of the truncated parameters, 
while the identifiability and cross-identifiability gramian uses linear combinations of parameters to be truncated.

\subsection{Combined Reduction}
The linear model in (\ref{linmod}) and the nonlinear model in (\ref{hypmod}) are next reduced in states and parameters employing the methods presented in section~\ref{combined} and section~\ref{joint}.
Controllability-based combined reduction uses the empirical sensitivity gramian for parameter reduction and balanced truncation for the state reduction.
Observability-based combined reduction uses the empirical identifiability gramian for parameter reduction and also balanced truncation for the state reduction.
The cross-gramian-based combined reduction utilizes the empirical joint gramian; the cross-identifiability gramian is used for the parameter reduction and the direct truncation of the cross gramian is used for the state reduction.
In figure~\ref{combolin} and figure~\ref{combonon} the relative $L_2$-error in the outputs $y$ is plotted for reduced orders $\dim(\tilde{x}) = 1 \hdots n-1$ and $\dim(\tilde{\theta}) = 1 \hdots n-1$.

\begin{figure}[h!]
\includegraphics[width=\textwidth]{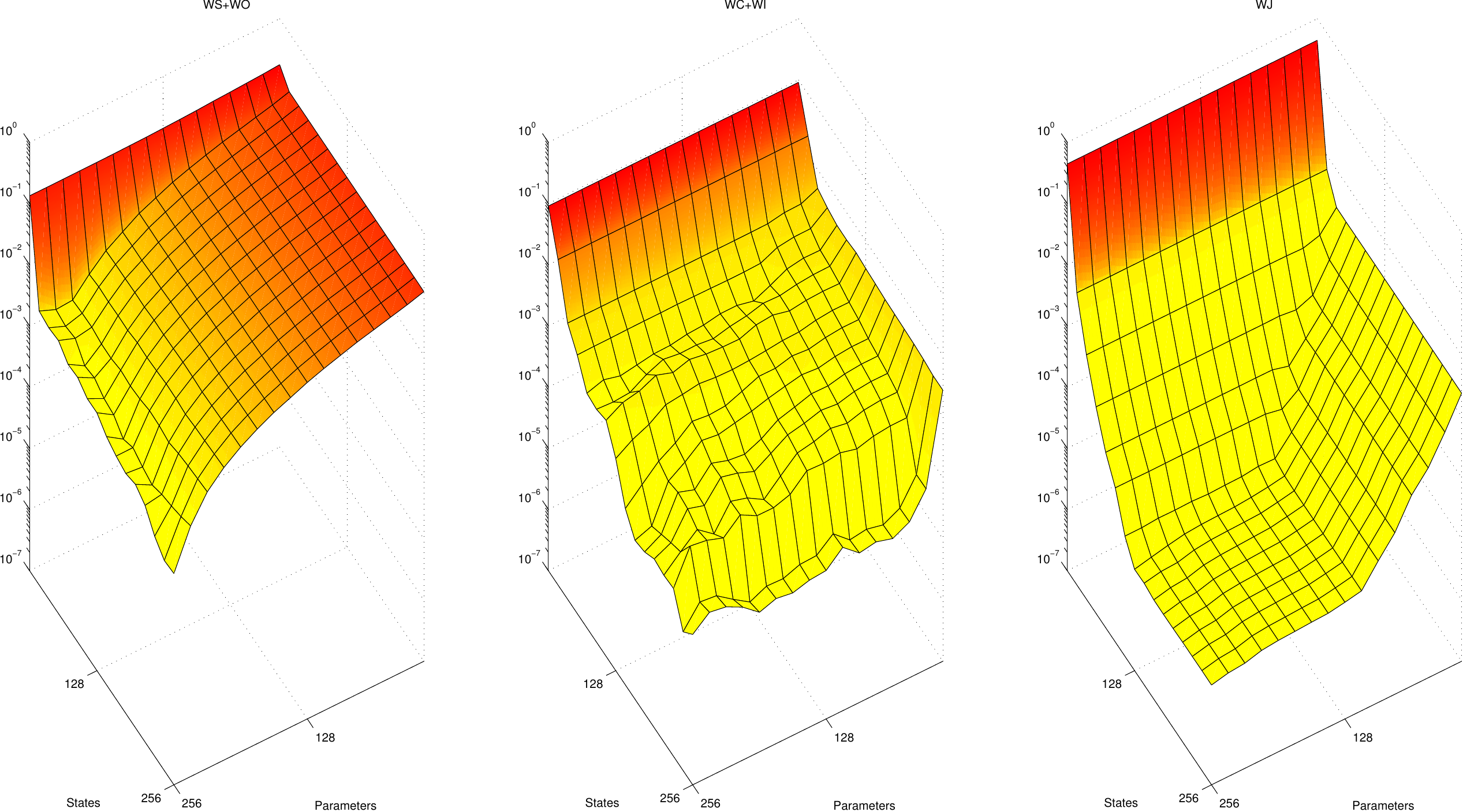}
\caption{Relative $L_2$ in the reduced order linear model after the combined reduction using controllability-based, observability-based and cross-gramian-based combined reduction.}
\label{combolin}
\end{figure}

\begin{figure}[h!]
\includegraphics[width=\textwidth]{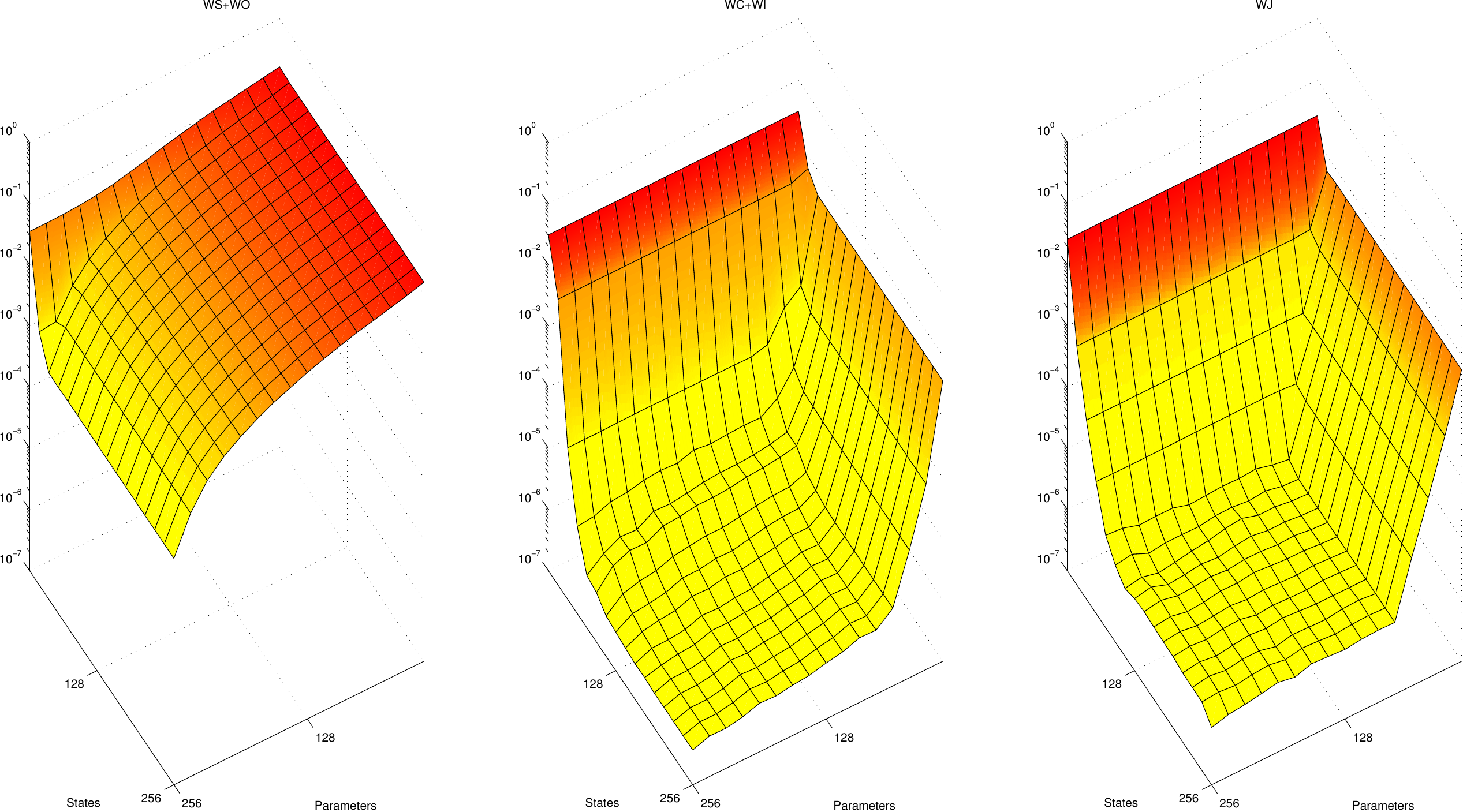}
\caption{Relative $L_2$ output error in the reduced order nonlinear model after the combined reduction using controllability-based, observability-based and cross-gramian-based combined reduction.}
\label{combonon}
\end{figure}

For a better comparison, a cross-section of the surfaces in figure~\ref{combolin} and figure~\ref{combonon} along the diagonals are plotted in figure~\ref{combored} showing the reduced order model's output error for the same reduced order in states and parameters $\dim(\tilde{x}) = \dim(\tilde{\theta})$.

\begin{figure}[h!]
\includegraphics[width=\textwidth]{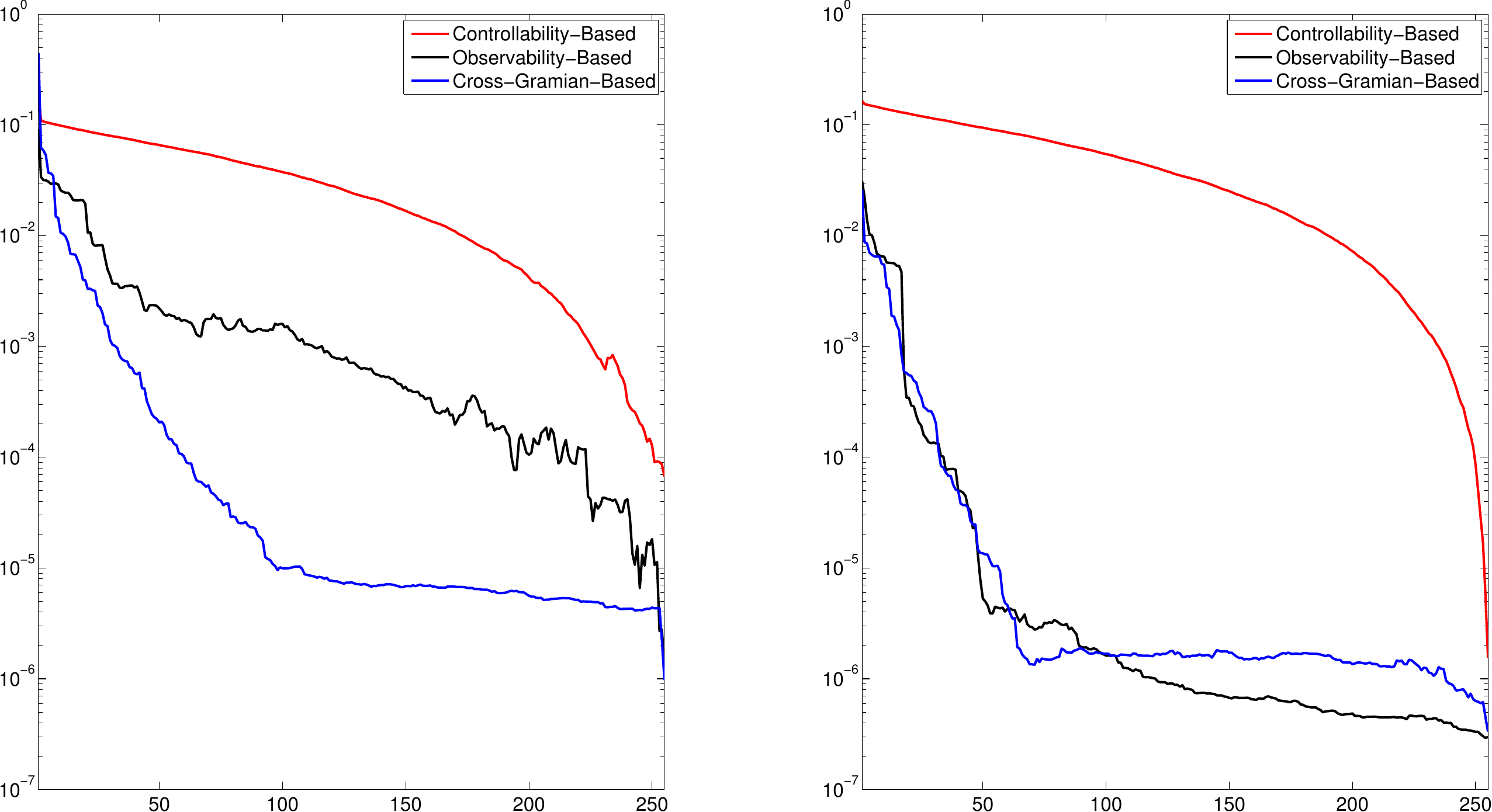}
\caption{Relative $L_2$ output error in the reduced order linear and nonlinear models after the combined reduction using controllability-based, observability-based and cross-gramian-based combined reduction for same reduced order of states and parameters.}
\label{combored}
\end{figure}

For all reduced order models obtained by combined state and parameter reduction, the parameter reduction error dominates the output error.
As for the parameter reduction, the controllability-based combined reduction by sensitivity gramian and balanced truncation performs worst with a slow descent in the output error.
In the linear setting the cross-gramian-based joint gramian performs significantly better than the combination of identifiability gramian and balanced truncation.
In the nonlinear setting the reduced order models of these methods exhibit similar behavior, which is due to the higher state reduction (see figure~\ref{statered}) error in the cross-gramian-based approach.

To assess the efficiency of the presented methods, the offline times\footnote{The time required to assemble the necessary empirical gramian matrices} are compared next.

\begin{figure}[h!]
\centering
\includegraphics[width=\textwidth]{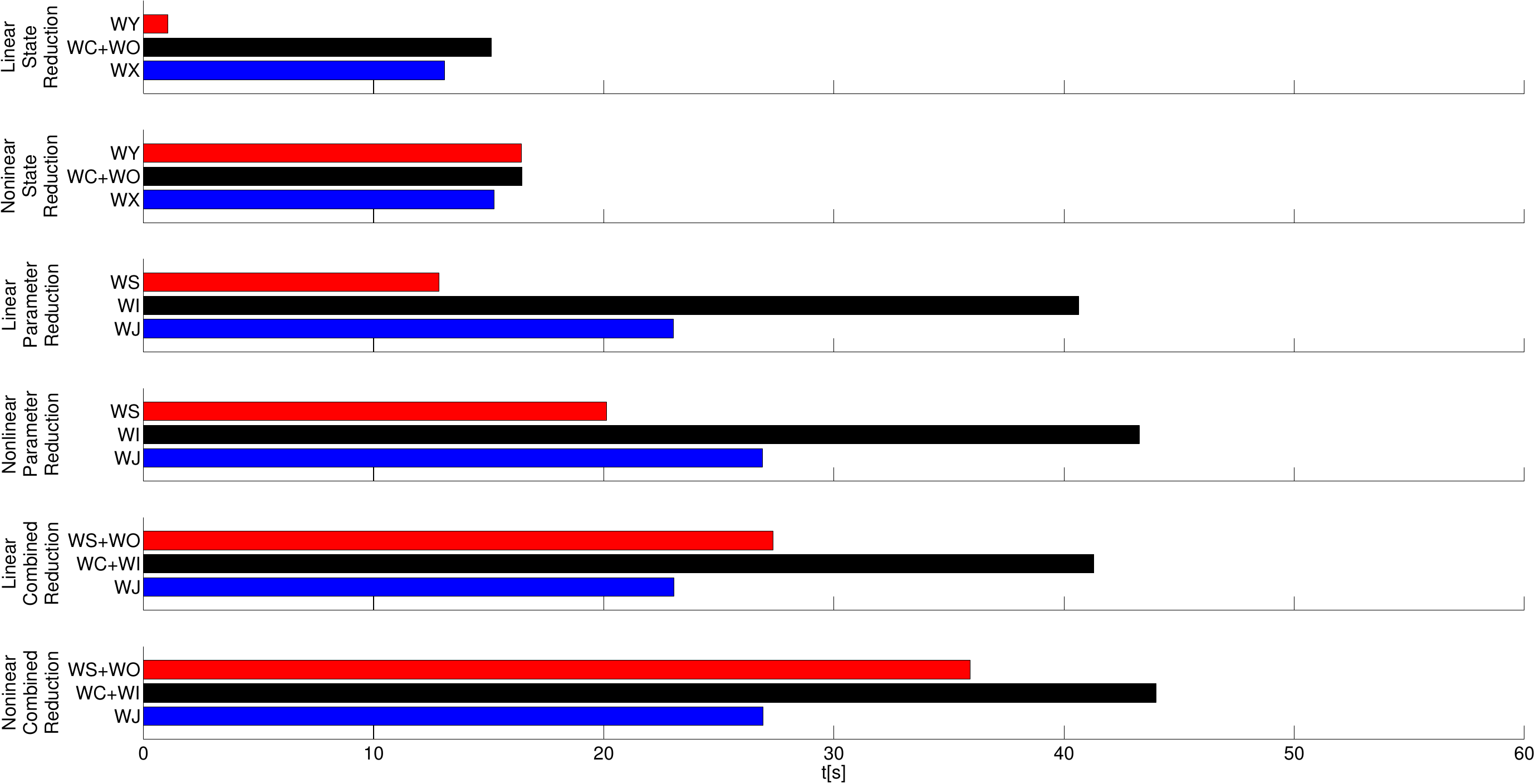}
\caption{Offline time for the empirical gramians and the associated decompositions for the state, parameter and combined reduction in the linear and nonlinear setting.}
\label{offtimes}
\end{figure}

The state reduction of the linear model is accomplished fastest by the balance-POD-related method, since the computational effort is required computing the equivalent of a single empirical controllability gramian.
For the nonlinear model this advantage is not existing, since no adjoint system for the nonlinear model is provided.
Notably, the cross gramian, for the direct truncation, is computed slightly faster than controllability, observability gramian and balancing transformation for the balanced truncation.

Among the empirical gramians for parameter reduction the sensitivity gramian is computed fastest, yet due to the high error this is the least applicable.
Between the identifiability and cross-identifiability gramian, which exhibit a comparable output error, the cross-gramian-based approach is considerably faster. 

For the combined reduction the offline times are similar to the offline times of the parameter reduction, with the exception of the controllability-based combined reduction,
which now takes longer than the cross-gramian-based joint gramian because of the additional observability gramian required for the balanced truncation state reduction.
Thus the empirical joint gramian is the fastest of the tested methods for combined reduction and provides a competitive output error.

\section{Conclusion}
In this paper the empirical cross gramian for MIMO systems and the empirical joint gramian for parameter and combined state and parameter reduction have been introduced and benchmarked.
The empirical cross gramian allows a state reduction of linear and nonlinear systems\footnote{For a transfer of the symmetry classification to nonlinear systems see \cite{ionescu11}}.
The empirical joint gramian not only enables cross-gramian-based parameter reduction, but also an efficient combined state and parameter reduction.
Both, the empirical cross gramian and the empirical joint gramian have been shown to be a viable alternative to balanced truncation and balancing-based combined reduction approaches. 

Further research has to be conducted on two-sided projections and error bounds for the reduced order models as well as on applying the empirical cross gramian to nonlinear or non-symmetric systems.
Existing extensions for non-symmetric systems are generalizing the symmetry constraint to orthogonal symmetry \cite{abreu86} or embedding into a symmetric system \cite{sorensen01}.

The empirical cross gramian for MIMO systems completes the set of empirical gramians for state reduction, while the joint gramian completes the body of parameter identification gramians and enables combined reduction without balancing.

\section*{Acknowledgement}
This work was supported by the Deutsche Forschungsgemeinschaft, the Open Access Publication Fund of the University of M\"unster, DFG EXC 1003 Cells in Motion - Cluster of Excellence, M\"unster, Germany as well as by the Center for Developing Mathematics in Interaction, DEMAIN, M\"unster, Germany.
\bibliographystyle{unsrt}
\bibliography{mpe}

\begin{thebibliography}{10}

\bibitem{moran07}
R.J. Moran, S.J. Kiebel, K.E. Stephan, R.B. Reilly, J.~Daunizeau, and K.J.
  Friston.
\newblock \bibdoi{A neural mass model of spectral responses in
  electrophysiology}{10.1016/j.neuroimage.2007.05.032}.
\newblock {\em NeuroImage}, 37(3):706--720, 2007.

\bibitem{moore81}
B.~Moore.
\newblock \bibdoi{Principal component analysis in linear systems:
  Controllability, observability, and model
  reduction}{10.1109/TAC.1981.1102568}.
\newblock {\em IEEE Transactions on Automatic Control}, 26(1):17--32, 1981.

\bibitem{kunisch99}
K.~Kunisch and S.~Volkwein.
\newblock \bibdoi{Control of the Burgers equation by a reduced-order approach
  using proper orthogonal decomposition}{10.1023/A:1021732508059}.
\newblock {\em Journal of Optimization Theory and Applications},
  102(2):345--371, 1999.

\bibitem{benner04}
P.~Benner.
\newblock \bibdoi{Solving large-scale control
  problems}{10.1109/MCS.2004.1272745}.
\newblock {\em Control Systems, IEEE}, 24(1):44--59, 2004.

\bibitem{haasdonk08}
B.~Haasdonk and M.~Ohlberger.
\newblock \bibdoi{Reduced basis method for finite volume approximations of
  parametrized linear evolution equations}{10.1051/m2an:2008001}.
\newblock {\em ESAIM: Mathematical Modelling and Numerical Analysis},
  42(2):277--302, 2008.

\bibitem{vandooren00}
P.M. Van~Dooren.
\newblock \biburl{Gramian based model reduction of large-scale dynamical
  systems}{citeseerx.ist.psu.edu/viewdoc/download?doi=10.1.1.29.2500&rep=rep1&type=pdf}.
\newblock {\em Numerical Analysis 1999, Research Notes in Mathematics},
  420:231--247, 2000.

\bibitem{laub87}
A.~Laub, M.~Heath, C.~Paige, and R.~Ward.
\newblock \bibdoi{Computation of system balancing transformations and other
  applications of simultaneous diagonalization
  algorithms}{10.1109/TAC.1987.1104549}.
\newblock {\em IEEE Transactions on Automatic Control}, 32(2):115--122, 1987.

\bibitem{antoulas05}
A.C. Antoulas.
\newblock {\em \bibdoi{Approximation of large-scale dynamical
  systems}{10.1137/1.9780898718713}}, volume~6.
\newblock Society for Industrial Mathematics, 2005.

\bibitem{fernando83a}
K.V. Fernando and H.~Nicholson.
\newblock \bibdoi{On the structure of balanced and other principal
  representations of SISO systems}{10.1109/TAC.1983.1103195}.
\newblock {\em IEEE Transactions on Automatic Control}, 28(2):228--231, 1983.

\bibitem{ionescu11}
T.C. Ionescu, K.~Fujimoto, and J.M.A. Scherpen.
\newblock \bibdoi{Singular value analysis of nonlinear symmetric
  systems}{10.1109/TAC.2011.2126630}.
\newblock {\em Transactions on Automatic Control of the IEEE},
  56(9):2073--2086, 2011.

\bibitem{lall99}
S.~Lall, J.E. Marsden, and S.~Glavaski.
\newblock \biburl{Empirical model reduction of controlled nonlinear
  systems}{authors.library.caltech.edu/20343/2/10.1.1.123.4669.pdf}.
\newblock {\em Proceedings of the IFAC World Congress}, F:473--478, 1999.

\bibitem{lall02}
S.~Lall, J.E. Marsden, and S.~Glavaski.
\newblock \bibdoi{A subspace approach to balanced truncation for model
  reduction of nonlinear control systems}{10.1002/rnc.657}.
\newblock {\em International Journal of Robust and Nonlinear Control},
  12(6):519--535, 2002.

\bibitem{hahn02a}
J.~Hahn and T.F. Edgar.
\newblock \bibdoi{Balancing approach to minimal realization and model reduction
  of stable nonlinear systems}{10.1021/ie0106175}.
\newblock {\em Industrial \& engineering chemistry research}, 41(9):2204--2212,
  2002.

\bibitem{streif06}
S.~Streif, R.~Findeisen, and E.~Bullinger.
\newblock \biburl{Relating cross gramians and sensitivity analysis in systems
  biology}{eprints.nuim.ie/1768/1/HamiltonGramian.pdf}.
\newblock {\em Theory of Networks and Systems}, 10.4:437--442, 2006.

\bibitem{streif09}
S.~Streif, S.~Waldherr, F.~Allg\"ower, and R.~Findeisen.
\newblock \biburl{Steady state sensitivity analysis of biochemical reaction
  networks: a brief review and new
  methods}{books.google.de/books?id=Haod3KR-tR8C&pg=PA129}.
\newblock In A.~Jayaraman and J.~Hahn, editors, {\em Methods in
  Bioengineering}, pages 129--148. Artech House {MIT} Press, 2009.
\newblock Systems Analysis of Biological Networks.

\bibitem{geffen08}
D.~Geffen, R.~Findeisen, M.~Schliemann, F.~Allg\"ower, and M.~Guay.
\newblock \bibdoi{Observability based parameter identifiability for biochemical
  reaction networks}{10.1109/ACC.2008.4586807}.
\newblock {\em Proceedings of the American Control Conference}, pages
  2130--2135, 2008.

\bibitem{sun06a}
C.~Sun and J.~Hahn.
\newblock \bibdoi{Parameter reduction for stable dynamical systems based on
  Hankel singular values and sensitivity analysis}{10.1016/j.ces.2006.04.027}.
\newblock {\em Chemical engineering science}, 61(16):5393--5403, 2006.

\bibitem{fernando82}
K.V. Fernando and H.~Nicholson.
\newblock \bibdoi{Minimality of SISO linear systems}{10.1109/PROC.1982.12460}.
\newblock {\em Proceedings of the IEEE}, 70(10):1241--1242, 1982.

\bibitem{fernando83b}
K.~V. Fernando and H.~Nicholson.
\newblock \bibdoi{On the Cauchy index of linear
  systems}{10.1109/TAC.1983.1103200}.
\newblock {\em IEEE Transactions on Automatic Control}, 28(2):222--224, 1983.

\bibitem{fernando84a}
K.V. Fernando and H.~Nicholson.
\newblock \bibdoi{On a fundamental property of the cross-Gramian
  matrix}{10.1109/TCS.1984.1085524}.
\newblock {\em IEEE Transactions on Circuits and Systems}, 31(5):504--505,
  1984.

\bibitem{fernando84b}
K.V. Fernando and H.~Nicholson.
\newblock \bibdoi{Reachability, observability, and minimality of MIMO
  systems}{10.1109/PROC.1984.13094}.
\newblock {\em Proceedings of the IEEE}, 72(12):1820--1821, 1984.

\bibitem{fernando85}
K.V. Fernando and H.~Nicholson.
\newblock \bibdoi{On the cross-Gramian for symmetric MIMO
  systems}{10.1109/TCS.1985.1085737}.
\newblock {\em IEEE Transactions on Circuits and Systems}, 32(5):487--489,
  1985.

\bibitem{laub83}
A.J. Laub, L.M. Silverman, and M.~Verma.
\newblock \bibdoi{A note on cross-Grammians for symmetric
  realizations}{10.1109/PROC.1983.12688}.
\newblock {\em Proceedings of the IEEE}, 71(7):904--905, 1983.

\bibitem{sorensen01}
D.C. Sorensen and A.C. Antoulas.
\newblock \biburl{Projection methods for balanced model
  reduction}{www.caam.rice.edu/caam/trs/2001/TR01-03.pdf}.
\newblock {\em Technical Report}, pages 1--18, 2001.

\bibitem{sorensen02}
D.C. Sorensen and A.C. Antoulas.
\newblock \bibdoi{The Sylvester equation and approximate balanced
  reduction}{10.1016/S0024-3795(02)00283-5}.
\newblock {\em Linear Algebra and its Applications}, 351:671--700, 2002.

\bibitem{baur08}
U.~Baur and P.~Benner.
\newblock \biburl{Cross-gramian based model reduction for data-sparse
  systems}{www.emis.ams.org/journals/ETNA/vol.31.2008/pp256-270.dir/pp256-270.pdf}.
\newblock {\em Electronic Transactions on Numerical Analysis}, 31:256--270,
  2008.

\bibitem{gheondea99}
A.~Gheondea and R.J. Ober.
\newblock \bibdoi{A Trace Formula For Hankel
  Operators}{10.1090/S0002-9939-99-04669-9}.
\newblock {\em Proceedings of the American Mathematical Society},
  127(7):2007--2012, 1999.

\bibitem{aldhaheri91}
R.~W. Aldhaheri.
\newblock \bibdoi{Model order reduction via real {S}chur-form
  decomposition}{10.1080/00207179108953642}.
\newblock {\em International Journal of Control}, 53(3):709--716, 1991.

\bibitem{moaveni06}
B.~Moaveni and A.~Khaki-Sedigh.
\newblock \bibdoi{Input-Output Pairing based on Cross-Gramian
  Matrix}{10.1109/SICE.2006.314989}.
\newblock {\em International Joint Conference SICE-ICAS}, pages 2378--2380,
  2006.

\bibitem{moaveni08}
B.~Moaveni and A.~Khaki-Sedigh.
\newblock \bibdoi{A new approach to compute the cross-Gramian matrix and its
  application in input-output pairing of linear multivariable
  plants}{10.3923/jas.2008.608.614}.
\newblock {\em Journal of Applied Sciences}, 8(4):608--614, 2008.

\bibitem{moaveni08b}
B.~Moaveni and A.~Khaki-Sedigh.
\newblock \bibdoi{Input-output pairing analysis for uncertain multivariable
  processes}{10.1016/j.jprocont.2007.10.015}.
\newblock {\em Journal of Process Control}, 18(6):527--532, 2008.

\bibitem{shaker12}
H.R. Shaker and M.~Komareji.
\newblock \bibdoi{Control Configuration Selection for Multivariable Nonlinear
  Systems}{10.1021/ie301137k}.
\newblock {\em Industrial and Engineering Chemistry Research},
  51(25):8583--8587, 2012.

\bibitem{hahn00}
J.~Hahn and T.F. Edgar.
\newblock \bibdoi{Reduction of nonlinear models using balancing of empirical
  gramians and Galerkin projections}{10.1109/ACC.2000.878734}.
\newblock {\em Proceedings of the American Control Conference}, 4:2864--2868,
  2000.

\bibitem{antoulas01}
A.C. Antoulas and D.C. Sorensen.
\newblock \biburl{Approximation of large-scale dynamical systems: an
  overview}{www.ece.rice.edu/~aca/mtns00.pdf}.
\newblock {\em International Journal of Applied Mathematics and Computer
  Science}, 11(5):1093--1121, 2001.

\bibitem{hahn02b}
J.~Hahn and T.F. Edgar.
\newblock \bibdoi{An improved method for nonlinear model reduction using
  balancing of empirical gramians}{10.1016/S0098-1354(02)00120-5}.
\newblock {\em Computers \& chemical engineering}, 26(10):1379--1397, 2002.

\bibitem{hahn03}
J.~Hahn, T.F. Edgar, and M.~Marquardt.
\newblock \bibdoi{Controllability and observability covariance matrices for the
  analysis and order reduction of stable nonlinear
  systems}{10.1016/S0959-1524(02)00024-0}.
\newblock {\em Journal of Process Control}, 13(2):115--127, 2003.

\bibitem{condon04}
M.~Condon and R.~Ivanov.
\newblock \bibdoi{Empirical balanced truncation of nonlinear
  systems}{10.1007/s00332-004-0617-5}.
\newblock {\em Journal of Nonlinear Science}, 14(5):405--414, 2004.

\bibitem{singh05}
A.K. Singh and J.~Hahn.
\newblock \bibdoi{On the use of Empirical Gramians for Controllability and
  Observability Analysis}{10.1109/ACC.2005.1469922}.
\newblock {\em Proceedings of the American Control Conference}, pages 140--146,
  2005.

\bibitem{hahn01}
J.~Hahn and T.F. Edgar.
\newblock \bibdoi{A gramian based approach to nonlinearity quantification and
  model classification}{10.1021/ie010155v}.
\newblock {\em Industrial \& engineering chemistry research},
  40(24):5724--5731, 2001.

\bibitem{condon04a}
M.~Condon and R.~Ivanov.
\newblock \bibdoi{Model reduction of nonlinear
  systems}{10.1108/03321640410510730}.
\newblock {\em COMPEL: The International Journal for Computation and
  Mathematics in Electrical and Electronic Engineering}, 23(2):547--557, 2004.

\bibitem{singh05a}
A.K. Singh and J.~Hahn.
\newblock \bibdoi{Determining optimal sensor locations for state and parameter
  estimation for stable nonlinear systems}{10.1021/ie040212v}.
\newblock {\em Industrial \& Engineering Chemistry Research},
  44(15):5645--5659, 2005.

\bibitem{eberle12}
C.~Eberle and C.~Ament.
\newblock \bibdoi{Identifiability and online estimation of diagnostic
  parameters with in the glucose insulin
  homeostasis}{10.1016/j.biosystems.2011.11.003}.
\newblock {\em Biosystems}, 107(3):135--141, 2012.

\bibitem{keil03}
A.~Keil and J.L. Gouz{\'e}.
\newblock \biburl{Model reduction of modular systems using balancing
  methods}{campar.in.tum.de/twiki/pub/Main/AndreasKeil/keil2003modelreduction.pdf}.
\newblock {\em Technical report, Munich University of Technology}, 2003.

\bibitem{himpe13a}
C.~Himpe and M.~Ohlberger.
\newblock \bibdoi{A Unified Software Framework for Empirical
  Gramians}{10.1155/2013/365909}.
\newblock {\em Journal of Mathematics}, 2013:1--6, 2013.

\bibitem{sun06b}
C.~Sun and J.~Hahn.
\newblock \bibdoi{Model reduction in the presence of uncertainty in model
  parameters}{10.1016/j.jprocont.2005.10.001}.
\newblock {\em Journal of Process Control}, 16(6):645--649, 2006.

\bibitem{wu13}
M.~Wu, B.~Yin, A.~Vosoughi, C.~Studer, J.~R. Cavallaro, and C.~Dick.
\newblock \bibdoi{Approximate Matrix Inversion for High-Throughput Data
  Detection in the Large-Scale MIMO Uplink}{10.1109/ISCAS.2013.6572301}.
\newblock {\em Proceedings of the 2013 IEEE International Symposium on Circuits
  and Systems}, pages 2155--2158.

\bibitem{emgr}
C.~Himpe.
\newblock {emgr - Empirical Gramian Framework}.
\newblock \url{http://gramian.de}.

\bibitem{octave}
{Octave community}.
\newblock {GNU Octave 3.8}.
\newblock \url{http://gnu.org/software/octave}, 2014.

\bibitem{matlab}
{MATLAB}.
\newblock {Version 8.3 (R2014a)}, 2014.

\bibitem{fortuna12}
L.~Fortuna and M.~Fransca.
\newblock {\em \biburl{Optimal and Robust Control: Advanced Topics with
  MATLAB}{http://books.google.de/books?id=WM3OzyHKlD4C}}.
\newblock Taylor \& Francis, 2012.

\bibitem{chen99}
Y.~Chen.
\newblock \biburl{Model order reduction for nonlinear
  systems}{http://hdl.handle.net/1721.1/9381}.
\newblock Master's thesis, Massachusetts Institute of Technology, 1999.

\bibitem{morwiki}
MORwiki Community.
\newblock {MORwiki - Model Order Reduction Wiki}.
\newblock \url{http://modelreduction.org}, 2014.

\bibitem{condon07}
M.~Condon.
\newblock \bibdoi{Model reduction of nonlinear
  systems}{10.1002/pamm.200701084}.
\newblock {\em Proceedings in Applied Mathematics and Mechanics},
  7(1):2130011--2130012, 2007.

\bibitem{liu98}
W.Q. Liu, V.~Sreeram, and K.L. Teo.
\newblock \bibdoi{Model reduction for state-space symmetric
  systems}{10.1016/S0167-6911(98)00024-3}.
\newblock {\em Systems \& Control Letters}, 34(4):209--215, 1998.

\bibitem{quan01}
Y.~Quan, H.~Zhang, and L.~Cai.
\newblock \bibdoi{Modeling and control based on a new neural network
  model}{10.1109/ACC.2001.946022}.
\newblock {\em Proceedings of the American Control Conference}, 3:1928--1929,
  2001.

\bibitem{willcox02}
K.~Willcox and J.~Peraire.
\newblock \bibdoi{Balanced model reduction via the proper orthogonal
  decomposition}{10.2514/2.1570}.
\newblock {\em AIAA Journal}, 40(11):2323--2330, 2002.

\bibitem{barbagallo08}
A.~Barbagallo, V.F. De~Felice, and K.K. Nagarajan.
\newblock \biburl{Reduced order modelling of a couette flow using Balanced
  Proper Orthogonal
  Decomposition}{https://dokumente.unibw.de/pub/bscw.cgi/S50cf6918/d2700941/BarbaDefelNagar.pdf}.
\newblock {\em ERCOFTAC}, 2008.

\bibitem{shaker12a}
H.R. Shaker.
\newblock \bibdoi{Generalized cross-gramian for linear
  systems}{10.1109/ICIEA.2012.6360824}.
\newblock {\em 7th IEEE Conference on Industrial Electronics and Applications},
  pages 749--751, 2012.

\bibitem{abreu86}
J.A. De~Abreu-Garcia and F.~Fairman.
\newblock \bibdoi{A note on cross Grammians for orthogonally symmetric
  realizations}{10.1109/TAC.1986.1104421}.
\newblock {\em IEEE Transactions on Automatic Control}, 31(9):866--868, 1986.

\end{thebibliography}

\end{document}